\documentclass[12pt]{amsart}
\usepackage{amssymb, amsmath}
\newtheorem{question}{Question}[section]
\newtheorem{definition}[question]{Definition}
\newtheorem{theorem}[question]{Theorem}
\newtheorem{corollary}[question]{Corollary}
\newtheorem{lemma}[question]{Lemma}
\newtheorem{example}[question]{Example}
\def\N{\mathcal{N}}
\def\G{\mathcal{G}}
\def\B{\mathcal{B}}

\def\L{\mathcal{L}}
\def\U{\mathcal{U}}
\def\V{\mathcal{V}}

\title{Covering by discrete and closed discrete sets}
\author{Santi Spadaro}
\address{
Department of Mathematics and Statistics\\
Auburn University\\
221 Parker Hall\\
Auburn, Alabama -- USA 36849-5310}

\email{spadasa@auburn.edu}

\thanks{Research partially supported by National Science Foundation grant DMS-0405216 (Principal Investigator - Dr. Gary Gruenhage)}



\begin{document}
\baselineskip.525cm

\begin{abstract}
Say that a cardinal number $\kappa$ is \emph{small} relative to the space $X$ if $\kappa <\Delta(X)$, where $\Delta(X)$ is the least cardinality of a non-empty open set in $X$. We prove that no Baire metric space can be covered by a small number of discrete sets, and give some generalizations. We show a ZFC example of a regular Baire $\sigma$-space and a consistent example of a normal Baire Moore space which can be covered by a small number of discrete sets. We finish with some remarks on linearly ordered spaces.
\end{abstract}

\maketitle

\section{Introduction}

We will assume all spaces to be Hausdorff. \emph{Crowded} is Eric Van Douwen's apt name for a space without isolated points. All undefined notions can be found in \cite{En}, \cite{G2} and \cite{J}. Let $dis(X)$ be the least number of discrete sets required to cover the space $X$. The cardinal function $dis(X)$ is introduced by Juh\'asz and Van Mill in \cite{JM}, where the authors provide some lower bounds for $dis(X)$ and ask whether $dis(X) \geq \mathfrak{c}$, for any crowded compact space $X$. Gruenhage \cite{G1} shows that this is the case, by proving that $dis(X)$ cannot be raised by perfect mappings. In \cite{JS} Juh\'asz and Szentmikl\'ossy prove that if $X$ is a compact space such that $\chi(x,X) \geq \kappa$ for every $x \in X$, then $dis(X) \geq 2^\kappa$, thus generalizing both Gruenhage's result and the classical \v Cech-Pospi\v sil theorem (in which the cardinality of $X$ takes the place of $dis(X)$). Let $\Delta(X)$ be the \emph{dispersion character of $X$}, that is, the least cardinality of a non-empty open set in $X$. Since in a compact space where every point has character at least $\kappa$ we have $\Delta(X) \geq 2^\kappa$, Juh\'asz and Szentmikl\'ossy ask the following natural question.

\begin{question} \cite{JS}
Is $dis(X) \geq \Delta(X)$ for any compact space $X$?
\end{question}

Our work on the above question led us to investigate for what kind of Baire spaces, other than the compact ones, Juh\'asz and Szentmikl\'ossy's inequality could be true. In this note we prove that $dis(X) \geq \Delta(X)$ for two classes of Baire generalized metric spaces which satisfy a mild separation-type property. Moreover, we construct examples of very good Baire spaces for which $dis(X) < \Delta(X)$.

In the last section we prove that $dis(X)=|X|$ for every locally compact Lindel\"of linearly ordered space (LOTS) and show an example of an hereditarily paracompact Baire LOTS for which the gap between $dis(X)$ and $\Delta(X)$ can be arbitrarily large.

\section{Generalized metric spaces}

Given a collection $\G$ of subsets of $X$, set $st(x, \G)=\bigcup \{
G \in \G: x \in G \}$ and $ord(x,\G)=|\{G \in \G: x \in G\}|$.
Recall that a sequence $\{\G_n : n \in \omega \}$ of open covers of
$X$ is said to be a \emph{development} if $\{st(x, \G_n): n \in
\omega \}$ is a local base at $x$ for every $x \in X$. A space is
called \emph{developable} if it admits a development. A regular
developable space is called a \emph{Moore space}.

\begin{definition}
Let $\kappa$ be a cardinal. We call a space \emph{$\kappa$-expandable} if every closed discrete set expands to a collection of open sets $\G$ such that $ord(x,\G) \leq \kappa$ for every $x \in X$.
\end{definition}

The following theorem is new even for all complete metric spaces.

\begin{theorem} \label{th1}
Let $X$ be a Baire $\omega_1$-expandable developable space. Then $dis(X) \geq \Delta(X)$.
\end{theorem}

\begin{proof}
Fix a development $\{\G_n: n \in \omega \}$ for $X$ and suppose by contradiction that $\tau=dis(X) < \Delta(X)$. Since the inequality $dis(X) \geq \omega_1$ is true for every crowded Baire space $X$ we can assume that $\tau \geq \omega_1$. Set $X=\bigcup_{\alpha < \tau} D_\alpha$, where each $D_\alpha$ is discrete. Define $D_{\alpha,n}=\{x \in D_\alpha : st(x, \G_n) \cap D_\alpha =\{x\}\}$ and set $X_n= \bigcup_{\alpha \in \tau} D_{\alpha,n}$.

\vspace{.1in}

\noindent \textbf{Claim:} For every $x \in X_k$ there is a neighbourhood $G$ of $x$ such that $|G \cap X_k| \leq \tau$.

\begin{proof}[Proof of Claim.]
Let $G \in \G_k$ be such that $x \in G$. Then $G$ hits each $D_{\alpha, k}$ in at most one point: indeed, if $y,z \in G \cap D_{\alpha, k}$ with $y \neq z$, we'd have both $st(y, \G_k) \cap D_{\alpha,k}=\{y\}$ and $z \in st(y,\G_k) \cap D_{\alpha,k}$, which is a contradiction.
\renewcommand{\qedsymbol}{$\triangle$}
\end{proof}

Now $X=\bigcup_{n \in \omega} X_n$, so, by the Baire property of $X$, there is $k \in \omega$ such that $U \subset \overline{X_k}$ for some non-empty open set $U$. By the claim we can assume that $|U \cap X_k| \leq \tau$. So $|U \cap (\overline{X_k} \setminus X_k) \cap D_{\alpha,j}| > \tau$ for some $\alpha < \tau$ and $j \in \omega$.

Notice that the set $D_{\alpha,j}$ is actually closed discrete: indeed suppose $y \notin D_{\alpha,j}$ were some limit point. Let $V \in \G_j$ be a neighbourhood of $y$ and pick two points $z, w \in V \cap D_{\alpha,j}$. By definition of $D_{\alpha,j}$ we have $st(z, \G_j) \cap D_{\alpha,j}=\{z\}$. But $w \in V \subset st(z,\G_j)$, which leads to a contradiction.

Observe now that also $S:=U \cap (\overline{X_k} \setminus X_k) \cap D_{\alpha,j}$ is closed discrete and hence we can expand it to a collection $\U=\{U_x : x \in S \}$ of open sets such that $ord(y, \U) \leq \omega_1$ for every $y \in X$. Set $V_x=U_x \cap st(x,\G_j) \cap U$ and observe that $V_x \neq V_y$ whenever $x \neq y$ and if we put $\V=\{V_x: x \in S\}$ then we also have that $ord(y,\V) \leq \omega_1$ for every $y \in X$. For every $x \in S$ pick $f(x) \in V_x \cap X_k$: the mapping $f$ has domain of cardinality $>\tau$, range of cardinality $\leq \tau$ and fibers of cardinality $\leq \omega_1$, which is a contradiction.
\end{proof}

\begin{corollary}
$dis(X) \geq \Delta(X)$, for every Baire collectionwise Hausdorff (or meta-Lindel\"of) developable space $X$.
\end{corollary}

\begin{corollary} \label{metr}
$dis(X) \geq \Delta(X)$, for every Baire metric space $X$.
\end{corollary}

Recall that a \emph{network} is a collection $\N$ of subsets of a topological space such that for every open set $U \subset X$ and every $x \in U$ there is $N \in \N$ with $x \in N \subset U$. A \emph{$\sigma$-space} is a space having a $\sigma$-discrete network.

Our next aim is proving that $dis(X) \geq \Delta(X)$ for every regular Baire $\omega_1$-expandable $\sigma$-space. We could give a more direct proof, but we feel that the real explanation for that is the following probably folklore fact, a proof of which can be found in \cite{vD}.

\begin{lemma} \label{douwen}
Every regular Baire $\sigma$-space has a dense metrizable $G_\delta$ subspace.
\end{lemma}

Call $dis^*(X)$ the least number of \emph{closed discrete} sets required to cover $X$. Clearly $dis(X) \leq dis^*(X)$. In a $\sigma$-space, one can use a $\sigma$-discrete network to split every discrete set into a countable union of closed discrete sets. So the following lemma is clear.

\begin{lemma}  \label{prop}
If $X$ is a crowded $\sigma$-space then $dis(X)=dis^*(X)$.
\end{lemma}

The next lemma and its proof are essentially due to the anonymous referee.

\begin{lemma} \label{lemref}
Let $X$ be an $\omega_1$-expandable crowded Baire space such that $dis^*(X) \leq \kappa$, and $A \subset X$ with $|A| \leq \kappa$. Then $|\overline{A}| \leq \kappa$.
\end{lemma}

\begin{proof}
Since $X$ is Baire crowded we can assume that $\kappa \geq \omega_1$. Let $X=\bigcup_{\alpha < \kappa} D_\alpha$, where each $D_\alpha$ is closed discrete. Let $B_\alpha = \overline{A} \cap D_\alpha$. Then $B_\alpha$ is closed discrete, so we may expand it to a family of open sets $\mathcal{U_\alpha}$ such that $ord(x, \U_\alpha) \leq \omega_1$ for every $x \in X$. Then $|\mathcal{U}_\alpha|=|B_\alpha|$ and for all $U \in \mathcal{U}_\alpha$, $U \cap A \neq \emptyset$. Fix some well-ordering of $A$ and define a function $f : \mathcal{U_\alpha} \to A$ by:

$$f(U)=\min \{a \in A : a \in U \}.$$

We have that $|f^{-1}(a)| \leq \aleph_1$ for every $a \in A$, and therefore $|B_\alpha|=|\mathcal{U}_\alpha| \leq |A| \cdot \aleph_1 \leq \kappa$.

Since $\overline{A}=\bigcup_{\alpha \in \kappa} B_\alpha$ it follows that $|\overline{A}| \leq \kappa$.
\end{proof}

The statement of the next theorem is due to the anonymous referee, and improves our original theorem where $X$ was assumed to be paracompact.

\begin{theorem} \label{th2}
Let $X$ be a regular $\omega_1$-expandable Baire $\sigma$-space.
Then $dis(X) \geq \Delta(X)$.
\end{theorem}

\begin{proof}
Fix some dense metrizable $G_\delta$ subspace $M \subset X$ and suppose by contradiction that $dis^*(X)=dis(X) < \Delta(X)$. Then Lemma $\ref{lemref}$ implies that $\Delta(M) \geq \Delta(X)$ and, since $M$ is Baire metric, by Corollary $\ref{metr}$ we have $dis(X) \geq dis(M) \geq \Delta(M)$. So $dis(X) \geq \Delta(X)$, and we are done.
\end{proof}

\begin{corollary} \label{cor}
For every paracompact Baire $\sigma$-space $X$ (in particular, for every stratifiable Baire space), we have $dis(X) \geq \Delta(X)$.
\end{corollary}

Notice that in the proofs of Theorems $\ref{th1}$ and $\ref{th2}$ all one needs is that $X$ be $dis(X)$-expandable.

Also, while we didn't use any separation other than Hausdorff in Theorem $\ref{th1}$, regularity seems to be essential in Theorem $\ref{th2}$, since one needs a $\sigma$-discrete network consisting of closed sets to prove Lemma $\ref{douwen}$. This suggests the following question.

\begin{question}
Is there a collectionwise Hausdorff or meta-Lindel\"of (non regular) Baire $\sigma$-space $X$ such that $dis(X) < \Delta(X)$?
\end{question}

\section{Good spaces with bad covers}

We now offer two examples to show that $\omega_1$-expandability is
essential in Theorem $\ref{th2}$. The first one is a modification of
an example of Bailey and Gruenhage \cite{BG}. We will need the
following combinatorial fact which slightly generalizes Lemma 9.23
of \cite{J}. It must be well-known, but we include a proof anyway
since we couldn't find a reference to it.

\begin{lemma}
Let $\kappa$ be any infinite cardinal. There is a family $\mathcal{A} \subset [\kappa]^{cf(\kappa)}$ of cardinality $\kappa^+$ such that $|A \cap B| < cf(\kappa)$ for every $A, B \in \mathcal{A}$.
\end{lemma}

\begin{proof}
We begin by showing that there is a family $\mathcal{F}$ of functions from $cf(\kappa)$ to $\kappa$ such that $|\mathcal{F}|=\kappa^+$ and $|\{\alpha \in cf(\kappa): f(\alpha)=g(\alpha) \}|< cf(\kappa)$, for any $f,g \in \mathcal{F}$. Indeed, suppose we have constructed $\{f_\alpha: \alpha < \kappa \}$ with the stated property. Let $\kappa=\sup_{\alpha < cf(\kappa)} \kappa_\alpha$. Define $f: cf(\kappa) \to \kappa$ in such a way that $f(\tau) \neq f_\alpha(\tau)$, for every $\alpha < \kappa_\tau$ and $\tau \in cf(\kappa)$. Fix $\alpha \in \kappa$: if $\tau < cf(\kappa)$ is such that $f(\tau)=f_\alpha(\tau)$ we must have $\kappa_\tau \leq \alpha < \kappa$. Hence $|\{\tau \in cf(\kappa): f(\tau)=f_\alpha(\tau) \}|<cf(\kappa)$.

Now for $\mathcal{A}$ we can take (on $cf(\kappa) \times \kappa$) the family of graphs of functions in $\mathcal{F}$.
\end{proof}

\begin{example} \label{ex1}
(ZFC) A regular Baire $\sigma$-space $P$ for which $dis(P) < \Delta(P)$.
\end{example}

\begin{proof}
Fix an almost disjoint family $\mathcal{A} \subset [\mathfrak{c}]^{cf(\mathfrak{c})}$ such that $|\mathcal{A}|=\mathfrak{c}^+$. For every partial function $\sigma \in \mathfrak{c}^{<\omega}$ such that $dom(\sigma)=k$ let $\L_\sigma=\{f_{\sigma,A} : A \in \mathcal{A} \}$ where $f_{\sigma,A}: cf(\mathfrak{c}) \to \mathfrak{c}^{<\omega}$ is defined as follows: $dom(f_{\sigma, A}(\alpha))=k+1$, $f_{\sigma,A}(\alpha) \upharpoonright k=\sigma$ for every $\alpha \in cf(\mathfrak{c})$ and $\{f_{\sigma,A}(\alpha)(k) : \alpha \in cf(\mathfrak{c}) \}$ is a faithful enumeration of $A$.

When $f \in \L_\sigma$ we will refer to $\rho_f=\sigma$ as the \emph{root} of $f$, and set $k_f=dom(\sigma)$.

Let now $L=\bigcup_{\sigma \in \mathfrak{c}^{<\omega}} \L_\sigma$ and $B=\mathfrak{\mathfrak{c}}^\omega$. We are going to define a topology on $P=B \cup L$ that induces on $B$ its natural topology. For every $\sigma \in \mathfrak{c}^{<\omega}$, let $[\sigma]=\{g \in B: g \supset \sigma \}$ and $$B(\sigma)=[\sigma] \cup \{f \in L: \rho_f \supseteq \sigma \}.$$ Let $\{A_n: n \in \omega \}$ be a partition of $\mathfrak{c}$ into sets of cardinality $\mathfrak{c}$.

For $f \in L$, $\delta \in cf(\mathfrak{c})$ and $k \in \omega$ let

$$B_{\delta, k}(f) =\{f\} \cup \bigcup_{\gamma > \delta} \left \{ B(f(\gamma)): f(\gamma)(k_f) \in \bigcup_{n > k} A_n \right \}.$$

The set $\mathcal{B}=\{B(\sigma), B_{\delta,k}(f) : \sigma \in \mathfrak{c}^{<\omega}, \delta \in cf(\mathfrak{c}), k \in \omega \}$ is a base for a topology on $P$, as items (2) and (3) in the following list of claims show.

\begin{enumerate}

\item For $\sigma_1, \sigma_2 \in \mathfrak{c}^{<\omega}$, $B(\sigma_1) \cap B(\sigma_2)=\emptyset$ if and only if $\sigma_1$ and $\sigma_2$ are incompatible.
\item Suppose $B(\sigma) \cap B_{\delta,k} (f) \neq \emptyset$. Then $\sigma \subseteq \rho_f$ or $\rho_f \subseteq \sigma$. If $\sigma \subseteq \rho_f$ then $B(\sigma) \cap B_{\delta,k}(f)=B_{\delta,k}(f)$. If $\sigma \supsetneq \rho_f$, then the intersection is $B(\sigma)$.

\item If $B_{\delta,j}(f) \cap B_{\delta',k}(g) \neq \emptyset$ and $\rho_g \subsetneq \rho_f$ then the intersection is either $B_{\delta, j}(f)$ or a set of the form $B(\sigma)$, for some $\sigma \in \{f(\gamma), g(\gamma'): \gamma > \delta, \gamma' > \delta' \}$.

\item If $B_{\delta,j}(f) \cap B_{\delta',k}(g) \neq \emptyset$ and $\rho_g=\rho_f$ then the intersection is a union of less than $cf(\mathfrak{c})$ sets of the form $B(\sigma)$ where $\sigma \in ran(f) \cap ran(g)$.
\end{enumerate}

\begin{proof}[Proof of items (1)-(4)]

Item (1) is easy. For item (2), observe that $B_{\delta,k}(f) \subseteq B(\rho_f)$, so $B(\rho_f) \cap B(\sigma) \neq \emptyset$ which implies that $\rho_f$ and $\sigma$ are compatible. If $\sigma \subseteq \rho_f$ then for each $\gamma > \delta$ we have $\sigma \subseteq f(\gamma)$ and $f \in B(\sigma)$, so $B_{\delta,k}(f) \subseteq B(\sigma)$.

If $\sigma \supsetneq \rho_f$ then let $\gamma > \delta$ be the unique ordinal such that $B(\sigma) \cap B(f(\gamma)) \neq \emptyset$. Since $\sigma$ and $f(\gamma)$ are compatible we must have $f(\gamma) \subset \sigma$, from which $B(\sigma) \subset B(f(\gamma))$ follows, and hence the claim.

To prove item (3) observe that if $B_{\delta,j}(f) \cap B_{\delta',k}(g) \neq \emptyset$ and $\rho_g \subsetneq \rho_f$ then $g \notin B_{\delta,j}(f)$ and, as the range of $f$ consists of pairwise incompatible elements we have that $[g(\tau)] \cap [\rho_f] \neq \emptyset$ for at most one $\tau \in cf(\mathfrak{c})$. Therefore, $B_{\delta,j}(f) \cap B_{\delta',k}(g)=B(g(\tau)) \cap B_{\delta,j}(f)$, and the rest follows from item (2).

Item (4) follows from \emph{almost-disjointness} of the ranges.
\end{proof}

\vspace{.1in}
\noindent \textbf{Claim 1:} The base $\mathcal{B}$ consists of clopen sets.

\begin{proof}[Proof of Claim 1.] To see that $B_{\delta,j}(f)$ is closed pick $g \in L \setminus B_{\delta,j}(f)$ and let $\gamma$ be large enough so that $f \notin B_{\gamma,j}(g)$. Suppose that $B_{\delta,j}(f) \cap B_{\gamma,j}(g) \neq \emptyset$. Then there are $\alpha>\delta$ and $\beta>\gamma$ such that $f(\alpha)$ and $g(\beta)$ are compatible. Now we must have $\rho_g=\rho_f$ or otherwise we would have either $\rho_f \supset g(\beta)$ and hence $f \in B_{\gamma,j}(g)$, or $\rho_g \supset f(\alpha)$, which would imply $g \in B_{\delta,j}(f)$. So, by item (4) we have $B_{\delta,j}(f) \cap B_{\gamma,j}(g)=\bigcup_{\tau \in C} B(g(\tau))$ where $|C| < cf(\mathfrak{c})$ and hence, if we let $\theta > \sup(C)$, then $B_{\theta,j}(g) \cap B_{\delta,j}(f)=\emptyset$.

Now, let $p \in B \setminus B_{\delta,j}(f)$ and $i=k_f+2$. We claim that $B(p \upharpoonright i) \cap B_{\delta,j}(f)=\emptyset$. Indeed, if that were not the case then $f(\gamma)$ and $p \upharpoonright i$ would be compatible, for some $\gamma$. So $f(\gamma) \subset p \upharpoonright i \subset p$, which implies $p \in B_{\delta,j}(f)$, contradicting the choice of $p$.

To see that $B(\sigma)$ is clopen, observe that $B$ is dense in $P$ and the subspace base is clopen, so we can restrict our attention to limit points of $B(\sigma)$ in $L$. Suppose that $f \in L \setminus B(\sigma)$ is some limit point, then, for all $\delta \in cf(\mathfrak{c})$ and all $j \in \omega$ we have $B_{\delta,j}(f) \cap B(\sigma) \neq \emptyset$. So $\rho_f$ and $\sigma$ are compatible; moreover $\rho_f \subsetneq \sigma$ or otherwise $f \in B(\sigma)$. Now there is at most one $\delta'$ such that $f(\delta')$ and $\sigma$ are compatible, whence the absurd statement $B_{\delta'+1,0}(f) \cap B(\sigma)=\emptyset$.
\renewcommand{\qedsymbol}{$\triangle$}
\end{proof}

\vspace{.1in}
\noindent \textbf{Claim 2:} $P$ is a $\sigma$-space.

\begin{proof}[Proof of Claim 2.]
For each $\sigma \in \mathfrak{c}^{< \omega}$ let $h(\sigma) \in \omega^{< \omega}$ be defined by $\sigma(i) \in A_j$ iff $h(\sigma)(i)=j$. For every $s \in \omega^{< \omega}$ put $\mathcal{B}_s=\{B(\sigma): h(\sigma)=s \}$. We claim that $\mathcal{B}_s$ is a discrete collection of open sets. Notice that the elements of $\mathcal{B}_s$ are all disjoint. Now if $x \in B \setminus \bigcup \B_s$, let $j=dom(s)$; then either $x \upharpoonright (j+1)$ extends (at most) one $\sigma$ such that $h(\sigma)=s$ or $x \upharpoonright (j+1)$ is incompatible with every such $\sigma$. So $B(x \upharpoonright (j+1))$ will hit at most one element of $\mathcal{B}_s$. If $f \in L$ then let $l=\max(ran(s))$: we claim that $B_{0,l}(f)$ hits at most one element of $\B_s$. Indeed, for fixed $\alpha$ such that $f(\alpha)(k_f) \in \bigcup_{n>l} A_n$ either $f(\alpha)$ is incompatible with every $\sigma$ such that $h(\sigma)=s$ or there is exactly one such $\sigma$ which is compatible with $f(\alpha)$. In the latter case we can't have $\sigma \supset \rho_f$ because $f(\alpha)(k_f) \notin ran(s)$, hence we have $\sigma \subset \rho_f$, which implies $B_{0,l} (f) \subset B(\sigma)$.

Now we claim that $L$ is a $\sigma$-closed discrete set. Indeed, for every $s \in \omega^{<\omega}$,  set $L_s=\{f \in L : h(\rho_f)=s \}$. If $g \in L_s$ then every fundamental neighbourhood of $g$ hits $L_s$ in the single point $g$. If $g \notin L_s$ then either $\rho_g$ is incompatible with every $\rho_f$ such that $f \in L_s$, in which case every fundamental neighbourhood of $g$ misses $L_s$, or there is $f \in L_s$ such that $\rho_g$ and $\rho_f$ are compatible. If $\rho_g \subsetneq \rho_f$ then let $l=s(k_g)$: we have $B_{0,l}(g) \cap L_s=\emptyset$. If $\rho_f \subset  \rho_g$, then the root of every function of $L$ which is in a fundamental neighbourhood of $g$ has domain strictly larger than $dom(s)$ and hence every fundamental neighbourhood of $g$ misses $L_s$.
\renewcommand{\qedsymbol}{$\triangle$}
\end{proof}

Observe now that $P$ is Baire, because $B \subset P$ is a dense Baire subset. Also, $dis(P)=\mathfrak{c} < \mathfrak{c}^+=\Delta(P)$
\end{proof}

One of the properties of Bailey and Gruenhage's example that was lost in the modification is first-countability. This suggests the following question.

\begin{question}
Is there in ZFC a first-countable regular $\sigma$-space $X$ for which $dis(X) < \Delta(X)$?
\end{question}

The reason why we insist on a ZFC example, is that we already have a consistent answer to the previous question. In fact, the space we are now going to exhibit is first-countable, normal and shows that $\omega_1$-expandability cannot be weakened to $\omega_2$-expandability in Theorem $\ref{th1}$. Our original motivation for constructing this example was showing that paracompactness could not be weakened to normality in Corollary $\ref{cor}$.

Recall that a \emph{$Q$-set} is an uncountable subset of a Polish space whose every subset is a relative $F_\sigma$, and a \emph{Luzin set} is an uncountable subset of a Polish space $P$ which meets every first category set of $P$ in a countable set. The existence of $Q$-sets and Luzin sets in the reals is known to be independent of ZFC (see, for example, \cite{M}). Fleissner and Miller \cite{FM} constructed a model of ZFC where there are a $Q$-set of the reals of cardinality $\aleph_2$ and a Luzin set of the reals of cardinality $\aleph_1$.

\begin{lemma} \label{lemma}
Let $C$ be some Polish space having a base $\B=\{B_n : n \in \omega \}$ such that $B_n$ is homeomorphic to $C$ for every $n \in \omega$. Given a $Q$-set of cardinality $\aleph_2$ in $C$, there is one which is dense and has dispersion character $\aleph_2$. Given a Luzin set in $C$, there is one which is locally uncountable and dense.
\end{lemma}

\begin{proof}
Let $X$ be a $Q$-set in $C$. Let $\B'=\{B \in \B: |B \cap X|<\aleph_2 \}$. Then $Y=X \setminus \bigcup \B'$ is a $Q$-set such that $\Delta(Y)=\aleph_2$. Set $n_0=0$ and let $Z_0$ be a homeomorphic copy of $Y$ inside $B_{n_0}$. Set $Z=Z_0$ and let $n_1$ be the least integer such that $B_{n_1} \cap Z =\emptyset$: clearly $n_1>n_0$. Now let $Z_1 \subset B_{n_1}$ be a homeomorphic copy of $Y$ and set $Z=Z_0 \cup Z_1$. Now suppose you have constructed a $Q$-set $Z$ such that $Z \cap B_i \neq 0$ for every $1 \leq i \leq n_{k-1}$ and let $n_k$ be the least integer such that $Z \cap B_{n_k}=\emptyset$; let $Z_k \subset B_{n_k}$ be a homeomorphic copy of $Y$ into $B_{n_k}$. At the end of the induction let $Z=\bigcup_{n \in \omega} Z_n$, then $Z$ is a $Q$-set with the stated properties. The second statement is proved in a similar way.
\end{proof}

\begin{example} \label{ex2}
A normal Baire Moore space $X$ for which $dis(X) < \Delta(X)$.
\end{example}

\begin{proof}
Take a model of ZFC where there are a Luzin set $L' \subset \mathbb{R}$ and a $Q$-set $Z \subset \mathbb{R}$ with the properties stated in Lemma $\ref{lemma}$. Let $f: \mathbb{R} \setminus \mathbb{Q} \to (\mathbb{R} \setminus \mathbb{Q})^2$ be any homeomorphism. Then $L=f(L' \setminus \mathbb{Q})$ is a Luzin subset of $(\mathbb{R} \setminus \mathbb{Q})^2$, and by Lemma $\ref{lemma}$ we can assume that it is locally uncountable and dense. Let $\mathbb{Q}=\{q_n: n \in \omega \}$ be an enumeration and set $Z_n=Z \times \{q_n\}$. Set $T=\bigcup_{n \in \omega} Z_n$ and define a topology on $X=L \cup T$ as follows: points of $L$ have neighbourhoods just as in the Euclidean topology on the plane, while a neighbourhood of a point of $x \in Z_n$ is a disk tangent at $x$ to $Z_n$, and lying in the upper half plane relative to that line.
Notice that $L$ is dense in $X$ so $X$ is a Baire space. Moreover $\Delta(X)=\aleph_2 > \aleph_1=dis(X)$.

To prove that $X$ is normal let $H$ and $K$ be disjoint closed sets. It will be enough to show that $H$ has a countable open cover, such that the closure of every member of it misses $K$ (see Lemma 1.1.15 of \cite{En}).  Fix $n \in \omega$. We have $H \cap Z_n=\bigcup_{j \in \omega} H_j$, where $H_j$ is closed in the Euclidean topology on $Z_n$ for every $j \in \omega$. Fix $j \in \omega$. For each $x \in H_j$ let $D(x, r_x)$ be a disk tangent to $Z_n$ at $x$ such that $D(x, r_x) \cap K=\emptyset$ and $r_x=\frac{1}{k}$ for some $k \in \omega$. Let $U=\bigcup_{x \in H_j} D(x, r_x)$. First of all, we claim that no point of $K \cap Z_n$ is in $\overline{U}$: indeed if $x \in K \cap Z_n$ then let $I_x$ be an interval containing $x$ and missing $H_j$, then the closest that a point of $H_j$ can come to $x$ is one of the endpoints of $I_x$ so there is room enough to separate $x$ from $U$ by a tangent disk.

Now $U=\bigcup_{n \in \omega} U_n$, where $U_n=\bigcup \{D(x,r_x): r_x=\frac{1}{n} \}$. Let $V_n=\bigcup \{D(x,\frac{r_x}{2}): r_x=\frac{1}{n} \}$. We claim that $\overline{V}_n \cap K \setminus Z_n=\emptyset$: indeed, if some point $x \in K \setminus Z_n$ were limit for $V_n$ then we would have a sequence of disks of radius $\frac{1}{2n}$ clustering to it. But then $x \in U_n$, which contradicts $U \cap K=\emptyset$.

To separate points of $H \setminus T$ from $K$ just choose for each such point an open set whose closure misses $K$ and use second countabiliy of $L$. That shows how to define the required countable open cover of $H$.

Finally, a development for $X$ is provided by $\G_n=\{D(x,n) : x \in X \}$ where $D(x,n)=B(x, \frac{1}{n}) \setminus  \bigcup_{i<n} Z_i$ if $x \in L$, while if $x \notin L$, $D(x,n)$ is a tangent disk of radius less than $\frac{1}{n}$ which misses $\bigcup \{Z_i: i<n$ and $x \notin Z_i \}$.
\end{proof}

The cardinal $\aleph_2$ can be replaced by any cardinal not greater than $\mathfrak{c}$, under proper set theoretic assumptions (see \cite{FM}). So the previous example shows that the gap between $dis(X)$ and $\Delta(X)$ for normal Baire Moore spaces can be as big as the gap between the first uncountable cardinal and the continuum.

Since normal Moore spaces are, consistently, metrizable, there is no chance of getting in ZFC a space with all the properties of Example $\ref{ex2}$. Nevertheless, the following question remains open.

\begin{question}
Is there in ZFC a normal Baire $\sigma$-space $X$ for which $dis(X) < \Delta(X)$?
\end{question}

Using a $Q$-set on a tangent disk space to get normality is an old trick (see for example \cite{T}). Also, to get a regular Baire Moore space $X$ for which $dis(X) < \Delta(X)$ it actually suffices to assume the negation of CH along with the existence of a Luzin set.

A potential way of weakening the set theoretic assumption in Example $\ref{ex2}$ would be to replace \emph{Luzin set} with \emph{Baire subset of cardinality $\aleph_1$}, but even such an object would be inconsistent with MA+ $\neg$ CH, while the presence of CH would make the whole construction worthless, so we have no clue even about the following.

\begin{question}
Is there, at least under MA+ $\neg$ CH or under CH, a normal Baire $\sigma$-space $X$ for which $dis(X) < \Delta(X)$?
\end{question}

Also, notice that no regular Baire $\sigma$-space $X$ for which $dis(X) < \Delta(X)$ can be separable under CH. That is because any regular separable space with points $G_\delta$ has cardinality $\leq \mathfrak{c}$ (fix any dense countable set $D$, then, the map taking any closed neighbourhod to its intersection with $D$ is 1-to-1. So there are no more than $\mathfrak{c}$ closed neighbourhods in the space, but every point in a regular space with $G_\delta$ points is the intersection of countably many closed neighbourhods). Thus $dis(X)=\aleph_1 \geq \Delta(X)$ if CH holds.

\section{Linearly ordered spaces}

Recall that a space is called a \emph{GO space} if it embeds in a LOTS. We denote by $m(X)$ the minimum number of metrizable spaces needed to cover $X$. The following result is due to Ismail and Szymanski.

\begin{lemma} \cite{IS}
Let $X$ be a locally compact Lindel\"of GO space. Then $w(X) \leq \omega  \cdot m(X)$
\end{lemma}

\begin{theorem}
Let $X$ be a locally compact Lindel\"of GO space. Then $dis(X)=|X|$.
\end{theorem}

\begin{proof}
Suppose by contradiction that there exists $\lambda <|X|$ such that $X=\bigcup \{ D_\alpha: \alpha \in \lambda \}$ where each $D_\alpha$ is discrete. Then $|D_\alpha| \leq w(X) \leq \omega \cdot m(X) \leq \lambda$, for every $\alpha \in \lambda$. So $|X| \leq \sup \{ |D_\alpha|: \alpha \in \lambda \} \cdot \lambda \leq \lambda < |X|$.
\end{proof}

In the previous theorem we cannot weaken locally compact Lindel\"of to paracompact Baire, as the following example shows. Recall that a space is called \emph{non-archimedean} if it has a base such that any two elements are either disjoint or one is contained in the other. Every non-archimedean space has a base which is a tree under reverse inclusion (see \cite{Ny}), and from this it is easy to see that it is (hereditarily) paracompact.

\begin{example}
There is a Baire non-archimedean (and hence hereditarily paracompact) LOTS $X$ such that $dis(X) < \Delta(X)$.
\end{example}

\begin{proof}
Let $\kappa$ and $\lambda$ be infinite cardinals such that $cf(\kappa) \leq \lambda$ but $\lambda <\kappa$. Let $\mathbb{W}=\{-1\} \cup \kappa$. Define an order on $\mathbb{W}$ by declaring $-1$ to be less than every ordinal. Let $X=\{f \in \mathbb{W}^{\lambda^+}: supp(f) < \lambda^+ \}$, where $supp(f)=\min \{\gamma < \lambda^+ : f(\alpha)=0$ for every $\alpha \geq \gamma\}$. Now take the topology induced on $X$ by the lexicographic order.

\vspace{.1in}

\noindent \textbf{Claim 1:}
$X$ is a strong Choquet space (and hence Baire).

\begin{proof} [Proof of Claim 1]
We are going to describe a winning strategy for player II in the strong Choquet game (see \cite{K}). In his first move player I chooses any open set $B_1$ and a point $f_1 \in B_1$. Player II then chooses points $a_1,b_1 \in X$ such that $f_1 \in (a_1, b_1) \subset B_1$. Let now $\alpha_1=\max \{supp(f_1), supp(a_1), supp(b_1) \}$ and $\overline{f_{\alpha_1}}=(f_1(\gamma): 0 \leq \gamma <\alpha_1)$. Define $f_1^-=\overline{f_{\alpha_1}}^\frown (-1, 0, \dots 0)$ and $f_1^+=\overline{f_{\alpha_1}}^\frown (1,0, \dots, 0)$.

Clearly $a_1<f_1^- < f_1 < f_1^+ <b_1$. Now in her first move player II chooses the open set $A_1=(f_1^-, f_1^+)$.

Player I responds by choosing any open set $B_2 \subset A_1$ and a point $f_2 \in B_2$. Player II proceeds as before. Notice that $f_{n+1}$ thus constructed agrees with $f_n$ up to $\alpha_n$ and that the point $h=\left (\, \bigcup \overline{f_{\alpha_n}}\, \right)^\frown (0,0, \dots, 0)$ is in $\bigcap_{n \geq 1} A_n$. So II has a winning strategy.
\renewcommand{\qedsymbol}{$\triangle$}
\end{proof}

\vspace{.1in}

\noindent \textbf{Claim 2:}
 $X$ is the union of $\lambda^+$ many discrete sets.

\begin{proof}[Proof of Claim 2.]
For every $\alpha \in \lambda^+$, let $D_\alpha=\{f \in X: supp(f)=\alpha \}$. Then $X=\bigcup_{\alpha \in \lambda^+} D_\alpha$ and each $D_\alpha$ is discrete. Indeed, let $f \in D_\alpha$ and define:

\begin{equation}
f^-(\beta)=\begin{cases}

           f(\beta) & \text{If $\beta < \alpha$} \\
           -1 & \text{If $\beta=\alpha$}\\
           0 & \text{If $\beta > \alpha$}
\end{cases}
\end{equation}

\noindent Similarly define:
\begin{equation}
f^+(\beta)=\begin{cases}

            f(\beta) & \text{If $\beta < \alpha$}\\
            1 & \text{If $\beta=\alpha$}\\
            0 & \text{If $\beta> \alpha$}
           \end{cases}
\end{equation}

\noindent Then $(f^-, f^+) \cap D_\alpha = \{f\}$.
\renewcommand{\qedsymbol}{$\triangle$}
\end{proof}

\vspace{.1in}

\noindent \textbf{Claim 3:}
$X$ is non-archimedean.

\begin{proof}[Proof of Claim 3.]
Let $\mathcal{B}=\{[\sigma] : \sigma \in \mathbb{W}^\alpha$ for some $\alpha \in \lambda^+\}$, where $[\sigma]=\{f \in X : \sigma \subset f \}$. Then $\mathcal{B}$ is a basis for our space. Every element of $\mathcal{B}$ is open: indeed, if $f \in [\sigma]$ then let $\alpha=\max \{dom(\sigma), supp(f)\}$ and $f^+$ and $f^-$ be defined as in the proof of Claim 2. Then $f \in (f^-, f^+) \subset [\sigma]$.

Now let $c \in (a,b)$. Then there are ordinals $\alpha$ and $\beta$ such that $a(\alpha) < c(\alpha)$, $c(\beta) < b(\beta)$, while $a(\gamma)=c(\gamma)$ and $c(\tau)=b(\tau)$ for every $\gamma < \alpha$ and every $\tau < \beta$. Set $\theta=\max \{\alpha, \beta \}+1$. We have that $[c \upharpoonright \theta] \subset (a,b)$.

Now given two elements of $\mathcal{B}$, either one is contained in the other, or they are disjoint. Therefore $X$ is non-archimedean.
\renewcommand{\qedsymbol}{$\triangle$}
\end{proof}

To complete the proof observe that $\Delta(X) \geq \kappa^{\lambda} > \kappa > \lambda^+ \geq dis(X)$.
\end{proof}

Since for fixed $\lambda$ there are arbitrarily big cardinals $\kappa$ having cofinality $\lambda$, the former example shows that the gap between $dis(X)$ and $\Delta(X)$ can be arbitrarily big for hereditarily paracompact Baire LOTS.

Notice that the Lindel\"of number of the previous space is $\geq \kappa$, in particular $X$ is never Lindel\"of.

\begin{question}
Is $dis(X) \geq \Delta(X)$ true for every (Lindel\"of, hereditarily paracompact) \v Cech complete LOTS $X$?
\end{question}

Finally, we would like to mention that we recently applied our result on metric spaces to give several partial answers to Juh\'asz and Szentmikl\'ossy's original question about compact spaces. They will be the subject of another paper.

\section{acknowledgements}

The author is greatly indebted to Gary Gruenhage for various stimulating discussion on the topic of this paper. The author would like to thank both Gary Gruenhage and the anonymous referee for their helpful remarks and suggestions that led to an improvement of both the content and the exposition of the paper. One final word of thanks to the referee for his/her careful proofreading of the manuscript.

\end{document}